\documentclass{article}
\usepackage{macros}

\bibliography{references}

\usepackage{subfiles}

\title{Maximal estimates for perturbations of the Schr\"odinger operator on $\T^d$}
\author{Inbo Gottlieb Fenves and Jia Hao Tan}

\date{\today}

\begin{document}

\maketitle

\begin{abstract}
We study $L^p_\xbf L^\infty_t$ maximal estimates for exponential sums associated to $C^2$ graph hypersurfaces, motivated by Schr\"odinger maximal estimates on $\T^d$. We show that the conjectured maximal estimate for the periodic Schr\"odinger equation fails for $C^2$-small perturbations of the paraboloid, which can be viewed as a higher-dimensional extension of \cite{FuRenWan}. Our approach uses new lower bounds for incidence estimates originally proven in \cite{CaiZha}, for which we provide an alternative proof based on homogeneous dynamics. We also show the estimates are essentially sharp at the decoupling endpoint for the paraboloid $p = \frac{2(d+2)}{d}$.
\end{abstract}

\section{Introduction}\label{section:intro}
In this paper, we consider $L^p_{\xbf} L_{t}^\infty$-maximal estimates for exponential sums of the form 
\begin{align*}
    u(\xbf,t) = \sum_{\qbf \in [Q]^{d}} b_\qbf e(\xbf \cdot \qbf/Q + t \psi(\qbf/Q)), \hspace{0.5cm} (\xbf,t) \in [0,Q]^d \times [0,Q^2]
\end{align*}
where $e(z) = e^{2\pi i z}$, $d \in \N$, $[Q] = \{1,...,Q\}$, $b_\qbf \in \C$, and $\psi \in C^2([0,1]^{d})$. We will be particularly interested in perturbations of the paraboloid and more generally in uniformly convex $C^2$ graph hypersurfaces. When $\psi(\xi) = |\xi|^2$, we can consider a solution $U(\xbf,t)$ to the Schr\"odinger equation on the unit torus
\begin{equation*}
    \begin{cases}
        i \partial_t U - \frac{1}{2\pi} \Delta U = 0 ,\quad (\xbf, t) \in \mathbb{T}^{d} \times [0,1]\\
        U(\xbf,0) = \sum_{\qbf \in [Q]^{d}} b_\qbf e(\xbf \cdot \qbf)
    \end{cases}
\end{equation*}
Hence, the exponential sum $u(\xbf,t) = U(\xbf/Q, t/Q^2)$ is the parabolically rescaled solution.  Maximal estimates have been studied in relation to the pointwise convergence of the Schr\"odinger equation on $\R^d$, first addressed by Carleson \cite{Car} in the Euclidean case for $x \in \R$ where he proved that $f \in H^{s}(\R)$ for $s \geq \frac{1}{4}$ was sufficient to guarantee 
\begin{align*}
    \lim_{t \to 0^{+}} e^{it\Delta} f(x)= f(x) \text{ for a.e. $x \in \R$}. 
\end{align*}
In $\R^1$, the necessity of $s \geq \frac{1}{4}$ was later proven by Dahlberg and Kenig \cite{DahKen}. For higher dimensions $\R^d$, Bourgain constructed a counterexample which proves that $s \geq \frac{d}{2(d+1)}$ is necessary (see \cite{Bou} and \cite{Pie} for further exposition), and the sufficiency of $s > \frac{d}{2(d+1)}$ was proved by Du, Guth, and Li \cite{DuGutLi} for dimension $d = 2$ and later by Du and Zhang \cite{DuZha} for all dimensions $d \geq 3$, thereby resolving the original problem of Carleson up to the endpoint. 

In contrast to the Euclidean problem, the pointwise convergence of the equation on even the one dimensional torus $\T^1$ is subtler due to the number theoretic nature of the Schr\"odinger kernel given by the Weyl sum
\begin{align*}
    K_Q(x,t) = \sum_{q \in [Q]} e\Big(q x + q^2t \Big), \quad (x,t) \in [0,1]^2
\end{align*}
which can be large on its major arcs (see \cite{Bar} for a more detailed discussion).  
Moyua and Vega \cite{Moy} showed that in the periodic case $\T^1$, $s \geq \frac{1}{4}$ was necessary and established sufficiency when $s > \frac{1}{3}$. Later, Compaan, Lucà, and Staffilani \cite{ComLucSta} extended the result to $\T^d$ and showed that $s \geq \frac{d}{2(d+1)}$ is necessary for pointwise convergence, while $s > \frac{d}{d+2}$ is sufficient. On $\T^d$, the pointwise convergence of the Schr\"odinger equation for general initial data $f \in H^s$ in the range $s \in \Big[\frac{d}{2(d+1)}, \frac{d}{d+2}\Big]$ remains open. The proofs of the sufficiency results have relied on periodic Strichartz estimates established by Bourgain and Demeter \cite{BouDem}. 

Define the \textit{critical} and \textit{conjectured} exponents $p_{\crit}$ and $p_{\con}$ respectively by
\begin{equation*}
    p_{\crit} = \frac{2(d+2)}{d} \quad \mathrm{and} \quad p_{\con} = \frac{2(d+1)}{d}.
\end{equation*}

In this paper, we work with perturbations of the parabolically rescaled sum for the Schr\"odinger equation. The sufficiency of $s > \frac{1}{p_{\con}}$ would match the Euclidean problem and is related to the following conjectured Schr\"odinger maximal estimate. 

\begin{conj}\label{conj:ConjUpperBound}
For all $2 \leq p \leq p_{\con}$, $Q \gg 1$, and $(b_\qbf)_{\qbf \in [Q]^d}$, we have
\begin{align}\label{eq:ConjUpperBound}
    \norm{\sup_{t \in [0,Q^2]} \sum_{\qbf \in [Q]^d} b_\qbf e\Big(\xbf \cdot \frac{\qbf}{Q} + t \frac{\vert \qbf \vert^2}{Q^2} \Big)}_{L_\xbf^{p} ([0,Q]^d)} \lesssim_\epsilon Q^{\frac{d}{p} +\frac{1}{p_{\con}} + \epsilon} \norm{b_\qbf}_{\ell^2([Q]^{d})}
\end{align}
\end{conj}

This conjectured maximal estimate would follow from the endpoint estimate at $p = p_{\con}$, where it reads
\begin{align}
    \norm{\sup_{t \in [0,Q^2]} \sum_{\qbf \in [Q]^d} b_\qbf e\Big(\xbf \cdot \frac{\qbf}{Q} + t \frac{\vert \qbf \vert^2}{Q^2} \Big)}_{L_\xbf^{p_{\con}} ([0,Q]^d)} \lesssim_\epsilon Q^{\frac{d}{2} + \epsilon} \norm{b_\qbf}_{\ell^2([Q]^{d})}.
\end{align}

One would hope that Conjecture \ref{conj:ConjUpperBound} could be settled through the decoupling methods established in \cite{BouDem}, but an immediate issue is the inability of decoupling to distinguish between the paraboloid and small perturbations of the paraboloid. In $d = 1$, Fu, Ren, and Wang \cite{FuRenWan} showed that the conjectured $L_x^4$ maximal estimate \eqref{eq:ConjUpperBound} failed over the more general class of \textit{uniformly convex} (or \textit{generalized Dirichlet} as in Fu, Guth, Maldague \cite{FuGutMal}) sequences falling under the decoupling regime. They showed there exists some uniformly convex sequence $(a_q)$, and initial data $(b_q) \subset \C$ such that 
\begin{align*}
    \norm{\sup_{t \in [0,Q^2]} \sum_{q \in [Q]} b_q e\Big(x \cdot \frac{q}{Q} + t a_q \Big)}_{L_x^{4} ([0,Q])} \gtrsim Q^{\frac{1}{2} + \frac{1}{12}} \Vert b_q \Vert_{\ell^2([Q])}.
\end{align*}
Moreover, they proved this estimate is sharp up to a $Q^\varepsilon$-loss following from the $\ell^2 L^6$-decoupling of the parabola \cite{BouDem}.

The conjectured maximal estimate remains open except in the special case of the Weyl sum when $b_\qbf \equiv 1$, where it was proved first for $d =1$ by Barron \cite{Bar}, and later Baker \cite{Bak} showed that one does not require an $Q^\varepsilon$ loss. The Weyl sum estimate was generalized to higher dimensions $d$ by Miao, Yuan, and Zhao \cite{Mia}. For $d =1$, Demeter \cite{Demeter2025LevelSet} was able to obtain essentially sharp $L_x^4$ Schrödinger maximal estimate for general initial data $b_\qbf$ without any Fourier analysis by studying the Schrödinger kernel more closely.

In this paper, we show the failure of the conjectured maximal estimate over the $C^2$ class proved in \cite{FuRenWan} for $d=1$ also holds in all dimensions \textit{perturbatively}, in the following sense.

\begin{thm}\label{thm:SharpSchrodinger}
    Suppose $\phi \in C^2([0,1]^{d})$ satisfies $\norm{\phi}_{C^2} \lesssim 1$. Then for all $\epsilon > 0$ and all $Q \gg_\epsilon 1$, there exists $\psi \in C^2$ which is $\epsilon$-close to $\phi$ in the $C^2$ topology, and a sequence $(b_\qbf)$ satisfying
    \begin{equation*}
        \norm{\sup_{t \in [0,Q^2]} \left| \sum_{\qbf \in [Q]^{d}} b_\qbf e\bigl(\xbf \cdot \qbf/Q + t\psi(\qbf/Q) \bigr) \right|}_{L_\xbf^p([0,Q]^{d})} \gtrsim_\epsilon Q^{\frac{d(d+1)}{2(d+2)}+\frac{1}{p}}\norm{b_\qbf}_{\ell^2}
    \end{equation*}
    for all $1 \le p < \infty$. In particular, we have
    \begin{equation*}
       \norm{\sup_{t \in [0,Q^2]} \left| \sum_{\qbf \in [Q]^{d}} b_\qbf e\bigl(\xbf \cdot \qbf/Q + t\psi(\qbf/Q) \bigr) \right|}_{L_\xbf^{p_{\crit}}([0,Q]^{d})} \gtrsim_\epsilon Q^{\frac{d}{2}}\norm{b_\qbf}_{\ell^2}
    \end{equation*}
    and
    \begin{equation*}
        \norm{\sup_{t \in [0,Q^2]} \left| \sum_{\qbf \in [Q]^{d}} b_\qbf e\bigl(\xbf \cdot \qbf/Q + t\psi(\qbf/Q) \bigr) \right|}_{L_\xbf^{p_{\con}}([0,Q]^{d})} \gtrsim_\epsilon Q^{\frac{d}{2} + \frac{d}{2(d+1)(d+2)}}\norm{b_\qbf}_{\ell^2}.
    \end{equation*}
\end{thm}

Theorem \ref{thm:SharpSchrodinger} follows from the more general Theorem \ref{thm:LowerMaximalBound}, which treats arbitrary codimension submanifolds. This is a generalization of \cite[Theorem 1.2]{FuRenWan} (up to the uniformity of scales) in two different directions. First, our result applies to arbitrary $C^2$-perturbations, as opposed to particularly constructed uniformly convex sequences; moreover, we are able to treat arbitrary dimension and codimension. In the other direction, decoupling of compact $C^2$-hypersurfaces gives us the following upper bounds. 
\begin{thm}\label{thm:UpperMaximalBound}
    Suppose $\phi \in C^2([0,1]^{d})$ satisfies $\Vert \varphi \Vert_{C^2} \lesssim 1$ and $cI \le \nabla^2\phi(\xi) \le CI$ for every $\xi \in [0,1]^d$. Then for any $(b_\qbf) \subset \C$, we have
    \begin{equation*}
        \norm{\sup_{t \in [0,Q^2]} \left| \sum_{\qbf \in [Q]^{d}} b_\qbf e\bigl(\xbf \cdot \qbf/Q + t\phi(\qbf/Q) \bigr) \right|}_{L_\xbf^p([0,Q]^{d})} \lesssim_\epsilon Q^{\frac{d}{p}+\frac{2}{p_{\crit}} + \epsilon} \norm{b_\qbf}_{\ell^2}
    \end{equation*}
    for all $2 \le p \le p_{\crit}$. In particular, we have
    \begin{equation*}
        \norm{\sup_{t \in [0,Q^2]} \left| \sum_{\qbf \in [Q]^{d}} b_\qbf e\bigl(\xbf \cdot \qbf/Q + t\phi(\qbf/Q) \bigr) \right|}_{L_\xbf^{p_{\crit}}([0,Q]^{d})} \lesssim_\epsilon Q^{\frac{d}{2} + \epsilon}\norm{b_\qbf}_{\ell^2}
    \end{equation*}
    and
    \begin{equation*}
        \norm{\sup_{t \in [0,Q^2]} \left| \sum_{\qbf \in [Q]^{d}} b_\qbf e\bigl(\xbf \cdot \qbf/Q + t\phi(\qbf/Q) \bigr) \right|}_{L_\xbf^{p_{\con}}([0,Q]^{d})} \lesssim_\epsilon Q^{\frac{d}{2}+\frac{d^2}{2(d+1)(d+2)} + \epsilon}\norm{b_\qbf}_{\ell^2}.
    \end{equation*}
\end{thm}

In dimension $d = 1$, we recover the $L_{x}^{p_{\con}} L_{t}^\infty$ estimate of \cite{FuRenWan} for both the upper and lower estimate, and in higher dimensions our maximal estimates are sharp up to an $Q^\varepsilon$ loss at $L_{\xbf}^{p_{\crit}} L_{t}^\infty$.

The primary tool in proving the lower bound of Theorem \ref{thm:SharpSchrodinger} is a result on intersections of $C^2$ submanifolds with the rescaled integer lattice proven by Cairo-Zhang.

\begin{thm}[\cite{CaiZha}, Theorem 6.1]\label{thm:CaiZha}
    Suppose $\phi \colon [0,1]^d \to \R^{n-d}$ is $C^2$ with $\norm{\phi}_{C^2} \lesssim 1$. Then for all $0 < \epsilon < 1$ and $Q \gg_\epsilon 1$, there exists a $C^2$ function $\psi$ which is $\epsilon$-close to $\phi$ in the $C^2$-topology and which satisfies
    \begin{equation*}
        \left| \left\{ \qbf \in \Z^{d} : \psi(\qbf/Q) \in Q^{-1}\Z^{n-d} \right\} \right| \gtrsim_\epsilon Q^{\frac{dn}{d+2(n-d)}}.
    \end{equation*}
\end{thm}

In Section \ref{section:counting}, we provide a proof of the following variation on Theorem \ref{thm:CaiZha}.

\begin{thm}\label{thm:PointsOnManifolds}
    Suppose $n \ge 3$. For all $\phi \in C^2([0,1]^d;\R^{n-d})$, $0 < \delta \ll_n 1$, $0 < \epsilon \ll_{n,\delta} 1$, and $Q \gg_{n,\delta,\epsilon} 1$, there exists $g \in \SL_n(\R)$ with $\max_{i,j}|g_{ij}-\delta_{ij}| < 2\delta$ and $\phi_g \in C^2([0,1]^d;\R^{n-d})$ with $\norm{\phi - \phi_g}_{C^2} \lesssim_{n} \epsilon$ for which
    \begin{equation*}
        |g\Mc_{\phi_g} \cap Q^{-1}\Z_{\prim}^n| \gtrsim_{n} \delta^{n^2-1}\epsilon^{n-d} Q^{\frac{dn}{d+2(n-d)}}.
    \end{equation*}
\end{thm}

Here $\Mc_\phi = \{(\xi,\phi(\xi)) : \xi \in [0,1]^d\}$ is the graph of $\phi$.

\begin{rmk}
    Theorem \ref{thm:PointsOnManifolds} requires no assumptions on the derivatives of $\phi$. These assumptions are only needed to guarantee that $g\Mc_{\phi_g}$ is still a graph over some $\Omega \subseteq \R^d$; see Proposition \ref{prop:graphs}. Moreover, all dependencies in Theorem \ref{thm:PointsOnManifolds} can be made explicit.
\end{rmk}

A corresponding statement is still true in the case $n=2$ (equivalently, for curves in $\R^2$) with weaker constants. See Theorem \ref{thm:PointsOnManifoldsn=2}.

\subsection{Sharpness of the bounds}
Here we make a comment on the lower and upper bounds. Let 
\begin{equation*}
    \alpha_{\con}(p) = \frac{d}{p} + \frac{1}{p_{\con}}, \quad \alpha_{\low}(p) = \frac{d(d+1)}{2(d+2)} + \frac{1}{p} \quad \mathrm{and} \quad \alpha_{\upp}(p) = \frac{d}{p}+\frac{2}{p_{\crit}}
\end{equation*}
denote the exponents of $Q$ arising from Conjecture \ref{conj:ConjUpperBound}, the lower bound in Theorem \ref{thm:SharpSchrodinger}, and the upper bound in Theorem \ref{thm:UpperMaximalBound} respectively.

\begin{figure}[H]
    \centering
    \includegraphics[scale=0.3]{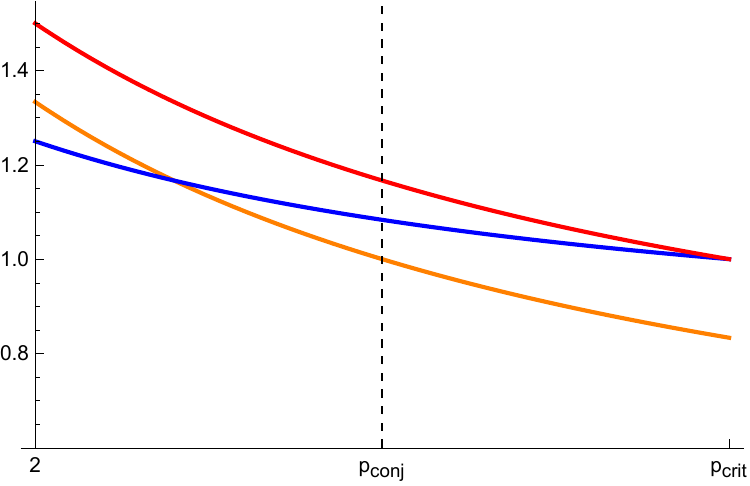}
    \caption{Graphs of $\alpha_{\con}$ (in \textcolor{orange}{orange}), $\alpha_{\upp}$ (in \textcolor{red}{red}), $\alpha_{\low}$ (in \textcolor{blue}{blue}) for $d=2$}
    \label{fig:placeholder}
\end{figure}
Let
\begin{align}
    D(p) &= \alpha_{\upp}(p) - \alpha_{\low}(p) = (d-1)\Big(\frac{1}{p} - \frac{1}{p_{\crit}}\Big), \hspace{0.5cm} 2\leq p \leq p_{\crit} \label{UBLB Gap} \\
    \Delta_{\low}(p) &= \alpha_{\low}(p) - \alpha_{\con}(p) = \frac{d(d^2 + d - 1)}{2(d+1)(d+2)} - \frac{d-1}{p}, \label{LBCB Gap} \\
    \Delta_{\upp}(p) &= \alpha_{\upp}(p) - \alpha_{\con}(p) = \frac{d^2}{2(d+1)(d+2)} > 0 \label{UBCB Gap},
\end{align}
hence in dimension $d= 1$, the gap from \eqref{UBLB Gap} is zero, and our estimates are sharp up to an $Q^\varepsilon$-factor for any exponent $2 \leq p \leq 6$, as was also shown in \cite{FuRenWan}. In higher dimensions, we have $\eqref{UBLB Gap}$ vanishing at the decoupling exponent $p = p_{\text{crit}}$ for general dimensions $d \geq 2$. We also have that $\alpha_{\low}(p) > \alpha_{\con}(p)$ if and only if $p > p_d = \frac{2(d-1)(d+1)(d+2)}{d(d^2 + d - 1)}.$ In particular, at the exponent $p_{\text{conj}}$, we have
\begin{align*}
    \Delta_{\low}(p_{\con}) = \frac{d}{2(d+1)(d+2)} > 0, \quad {\text{and }}\lim_{d \to \infty} p_d = 2.
\end{align*}
Hence the maximal estimate of Conjecture \ref{conj:ConjUpperBound} for the paraboloid at $p_{\text{conj}}$ fails over the larger class of $C^2$ hypersurfaces.

\begin{que}
    What are the correct values of $\alpha_{\upp}$ for $2 \le p < p_{\crit}$ for a general convex $C^2$-surface? 
\end{que}

\begin{rmk}
    We remark that the matching lower and upper bounds of Theorems \ref{thm:SharpSchrodinger} and \ref{thm:UpperMaximalBound} at the critical exponent $p_{\crit}$ yield a new proof of the classical bound
    \begin{equation}\label{eq:pointsonahypersurface}
        |\Sigma \cap Q^{-1}\Z^{d+1}| \lessapprox Q^{\frac{d(d+1)}{d+2}}
    \end{equation}
    for $C^2$ strictly convex hypersurfaces $\Sigma \subseteq \R^{d+1}$, as well as its optimality. The upper bound of \eqref{eq:pointsonahypersurface} was originally proven by Andrews\cite{And}. We also refer the reader to the work of Kiyohara \cite{Kiyohara2024LatticePoints} for recent lattice point counting results using the decoupling of the moment curve.
\end{rmk}

\subsection{Structure of the paper}

In Section \ref{section:lower}, we prove Theorem \ref{thm:SharpSchrodinger} using lower bounds on incidences between lattices and submanifolds. In Section \ref{section:upper}, we prove Theorem \ref{thm:UpperMaximalBound} via $\ell^2$-decoupling \cite{BouDem}. The incidence estimate Theorem \ref{thm:PointsOnManifolds} is shown in Section 4, and in Section 5 we modify this argument for the case $n=2$.

\textbf{Acknowledgments} The second author was supported by the National Science Foundation under grant No. DMS-2037851.

\section{Lower Bounds}\label{section:lower}

In this section, we prove Theorem \ref{thm:SharpSchrodinger} assuming Theorem \ref{thm:CaiZha}. Suppose throughout that $n \ge 2$, $n/2 \le d < n$, and $m=n-d$. Let $\phi \colon [0,1]^d \to \R^{m}$ be a $C^2$ function satisfying $\norm{\phi}_{C^2} \lesssim 1$.

\begin{thm}\label{thm:LowerMaximalBound}
    For all $\epsilon > 0$, there exists $Q_0 \gg_\epsilon 1$ so that the following holds. For all $Q \ge Q_0$, there exists some $\psi$ which is $\epsilon$-close to $\phi$ in the $C^2$ topology and a sequence $(b_\qbf)$ for which
    \begin{equation*}
        \norm{\sup_{\tbf \in [0,Q^2]^{m}} \left| \sum_{\qbf \in [Q]^{d}} b_\qbf e\bigl(\xbf \cdot \qbf/Q + \tbf \cdot \psi(\qbf/Q) \bigr) \right| }_{L_\xbf^p([0,Q]^d)} \gtrsim_\epsilon Q^{\frac{dn}{2(d+2m)}+\frac{m}{p}}\norm{b}_{\ell^2}.
    \end{equation*}
\end{thm}

\begin{proof}
    To start, let us write $\R^d = \R^{m} \times \R^{d-m}$, with associated coordinates $\xbf = (\xbf',\xbf'')$.
    
    Now, set $\phi_0(\xbf) = \phi(\xbf) + \xbf'/Q$. By Theorem \ref{thm:CaiZha}, for all $Q \gg_\epsilon 1$, there exists a $\psi_0$ which is $\epsilon$-close to $\phi_0$ for which the set
    \begin{equation*}
        B = \left\{ \qbf \in [Q]^d : \psi_0(\qbf/Q) \in Q^{-1}\Z^{m} \right\}
    \end{equation*}
    satisfies $|B| \gtrsim_\epsilon Q^{\frac{dn}{d+2m}}$. 
    
    Let us set $\psi(\xbf) = \psi_0(\xbf) - \xbf'/Q$ and $b_\qbf = \ind_B(\qbf)$. Clearly $\norm{\phi - \psi}_{C^2} \lesssim \epsilon$. We also let
    \begin{equation*}
        u(\xbf,\tbf) = \sum_{\qbf \in B} e\bigl( \xbf \cdot \qbf/Q + \tbf \cdot \psi(\qbf/Q) \bigr).
    \end{equation*}
    Then for $\pbf \in [Q]^d$, we have
    \begin{align*}
        u(\pbf,Q\pbf') &= \sum_{\qbf \in B} e\bigl( \pbf \cdot \qbf/Q + Q\pbf' \cdot \psi(\qbf/Q) \bigr) \\
        &= \sum_{\qbf \in B} e\bigl( \pbf \cdot \qbf/Q + Q\pbf' \cdot \psi_0(\qbf/Q) - Q\pbf' \cdot \qbf'/Q^2 \bigr) \\
        &= \sum_{\qbf \in B} e\bigl( \pbf'' \cdot \qbf''/Q \bigr)e\bigl( Q\pbf' \cdot \psi_0(\qbf/Q) \bigr) \\
        &= \sum_{\qbf \in B} e\bigl( \pbf'' \cdot \qbf''/Q \bigr),
    \end{align*}
    since $Q\pbf' \cdot \psi_0(\qbf/Q) \in \Z$ by construction of $B$. Thus if $\pbf'' = 0$, we have
    \begin{equation*}
        u(\pbf, Q\pbf') = |B|.
    \end{equation*}
    Moreover, we see that whenever $|(\xbf,\tbf) - (\pbf,Q\pbf')| \ll 1$, we have
    \begin{equation*}
        \Re(u(\xbf,\tbf)) \gtrsim |B|.
    \end{equation*}
    Thus, we have
    \begin{equation*}
        \norm{\sup_{\tbf \in [0,Q^2]^{m}} \left| u(\xbf,\tbf) \right| }_{L_\xbf^p([0,Q]^d)} \gtrsim |B| Q^{m/p} \gtrsim_\epsilon  Q^{\frac{dn}{2(d+2m)}+\frac{m}{p}}\norm{b}_{\ell^2},
    \end{equation*}
    as desired.
\end{proof}

From here, we easily deduce Theorem \ref{thm:SharpSchrodinger}.

\begin{proof}[Proof of Theorem \ref{thm:SharpSchrodinger}]
    Apply Theorem \ref{thm:LowerMaximalBound} in the case $d =n-1$, $m=1$; then
    \begin{equation*}
        \frac{dn}{2(d+2m)}+\frac{m}{p} = \frac{d(d+1)}{2(d+2)} + \frac{1}{p}.
    \end{equation*}
\end{proof}

\subsection{Some remarks}

\begin{rmk}
    Since the bound in Theorem \ref{thm:SharpSchrodinger} will later be shown to be sharp up to $\lesssim Q^\varepsilon$-loss at $p_{\crit}$, let us examine the error term in the proof of Theorem \ref{thm:LowerMaximalBound} where $\pbf'' \ne 0$, in the case $d=n-1$. We are tasked with understanding the sum
    \begin{equation*}
        \left(\sum_{\pbf'' \ne 0} \left| \sum_{\qbf \in B} e(\pbf'' \cdot \qbf''/Q) \right|^p\right)^{1/p}.
    \end{equation*}
    Since $B$ is a ``random'' set, we would expect square-root cancellation whenever $\pbf'' \ne 0$:
    \begin{equation*}
        \left| \sum_{\qbf \in B} e(\pbf'' \cdot \qbf''/Q) \right| \lessapprox \sqrt{|B|},
    \end{equation*}
    in which case
    \begin{equation*}
        \left( \sum_{\pbf'' \ne 0} \left| \sum_{\qbf \in B} e(\pbf'' \cdot \qbf''/Q) \right|^{p_{\crit}} \right)^{1/p_{\crit}} \lessapprox Q^{\frac{d-1}{p_{\crit}}} \norm{b}_{\ell^2} = Q^{\frac{d}{2}\frac{d-1}{d+2}}\norm{b}_{\ell^2},
    \end{equation*}
    which is dominated by the main term $Q^{\frac{d}{2}}\norm{b}_{\ell^2}$ coming from $\pbf'' = 0$, as is necessary for the corresponding upper bound to hold.
\end{rmk}

\begin{rmk}
    For the case of higher codimension $\frac{n}{2} \le d \le n$, it is unclear to us whether or not Theorem \ref{thm:LowerMaximalBound} should be expected to be sharp, or even what the correct upper bound to conjecture would be (in any particular category of $C^2$-submanifolds). 
    
    For example, an application of $L^6\ell^6$-decoupling for surfaces in $\R^4$ as in Bourgain-Demeter \cite{BouDem1}, using the same strategy as Section \ref{section:upper}, yields the upper bound\footnote{Note that \cite{BouDem1} is proven for $C^3$-surfaces in $\R^4$, and need not apply for $C^2$.}
    \begin{equation*}
        \norm{\sup_{\tbf \in [0,Q^2]^2} u(\xbf,\tbf)}_{L_\xbf^6([0,Q]^2)} \lessapprox Q^{15/9} \norm{b}_{\ell^6}.
    \end{equation*}
    On the other hand, Theorem \ref{thm:LowerMaximalBound} yields a sequence $(b_\qbf)$ for which
    \begin{equation*}
        \norm{\sup_{\tbf \in [0,Q^2]^2} u(\xbf,\tbf)}_{L_\xbf^6([0,Q]^2)} \gtrsim Q^{13/9} \norm{b}_{\ell^6}.
    \end{equation*}
    
    The case $d < \frac{n}{2}$ is even more mysterious, and we are unsure how best to utilize the additional degrees of freedom to ``cancel out'' oscillation in the phase.
\end{rmk}

\section{Upper Bound}\label{section:upper}
The main result of this section is the upper bound of Theorem \ref{thm:UpperMaximalBound}:
\begin{equation}
        \norm{\sup_{t \in [0,Q^2]} \left| \sum_{\qbf \in [Q]^{d}} b_\qbf e\bigl(\xbf \cdot \qbf/Q + t\psi(\qbf/Q) \bigr) \right|}_{L_\xbf^p([0,Q]^{d})} \lesssim Q^{\frac{d}{p}+\frac{d}{d+2} + \epsilon}\norm{b}_{\ell^2}, \hspace{.3cm} 2 \le p \le p_{\crit}.
\end{equation}

The estimate will pass through a localized version of the global $\ell^2L^{p_{\crit}}$-decoupling inequality of compact $C^2$-hypersurfaces (\cite{BouDem}, Theorem 1.1); the local theorem can be found in the subsequent study guide (\cite{bourgain2017study}, Theorem 5.1).

We first establish some notation on weights: 
Define the weight adapted to the unit cube
\[w(\zbf) = \Big(1+ \text{dist}\Big(\zbf, \Big[-\frac{1}{2},\frac{1}{2}\Big]^{d+1}\Big) \Big)^{-100d}
\]
and for a unit cube $Y$ centered at $\zbf_Y$, we define the weight $w_Y$ adapted to $Y$ by
\[
w_Y(\zbf) = w(\zbf - \zbf_Y).
\]
For a general set $X \subset \R^{d+1}$, we can take $\mathcal{Y}(X)$ to be minimal unit lattice cube cover of $X$ and define the weight adapted to $X$ by $w_X(\zbf) = \sum_{Y \in \mathcal{Y}(X)} w_Y(\zbf)$.
We express the weighted norm of $f$ by
\[
\Vert f \Vert_{L^p(w_{X})} = \Big(\int_{\R^{d+1}} \vert f(\zbf) \vert^p w_X(\zbf) d\zbf\Big)^{\frac{1}{p}}.
\]
\begin{thm}[Local Decoupling, \cite{bourgain2017study}]\label{BD-Decoupling}
Let $\psi: [0,1]^{d} \to \R$ be a $C^2$-function with $\Vert \psi \Vert_{c_2} \lesssim 1$ and $cI \leq \nabla^2 \psi \le CI$. Consider the hypersurface $\Sigma_\psi = \{(\xi, \psi(\xi)) : \xi \in [0,1]^d\} \subset \R^{d+1}$. For a large dyadic integer $R \in 4^{\N}$, partition $[0,1]^{d} = \bigsqcup \theta$ into cubes $\theta$ of length $\ell(\theta) = R^{-\frac{1}{2}}$. We denote by
\[
\mathcal{N}_{R^{-1}}(\theta):= \{(\xi, \psi(\xi)+ y): \xi \in \theta,\  y \in [-R^{-1},R^{-1}]\}
\]
the $R^{-1}$-neighborhood of $\theta$ over $\Sigma_\psi$. 
For any function $u: \R^{d+1} \to \C$ with Fourier support in $\mathcal{N}_{R^{-1}}([0,1]^{d})$, we define 
\[
\hat u_\theta = \hat u \mathbf{1}_{\mathcal{N}_{R^{-1}}(\theta)}.
\]
Then for any $\epsilon > 0$ and $2 \leq p \leq p_{\mathrm{crit}} = \frac{2(d+2)}{d}$, there exists a constant $C_{\epsilon, p , d} > 0$ such that the following inequality holds:
    \begin{align}
        \Vert u \Vert_{L^p(w_{B_{R}^{d+1}})} \leq C_{\epsilon,p,d,\psi} R^\epsilon \Big(\sum_{\theta} \Vert u_\theta \Vert^{2}_{L^p(w_{B_{R}^{d+1}})}\Big)^\frac{1}{2}.
    \end{align}
\end{thm}

We will also note the local constancy property which is exploited frequently in harmonic analysis. 

\begin{prop}[\cite{FuRenWan} Prop 2.2, \cite{demeter2020fourier}]\label{local constancy}
    If $f$ has Fourier transform supported in $B_{100}(0)$ and $Y$ is a unit cube, then for $1 \leq r < s \leq \infty$,
    we have 
    \begin{align*}
        \Vert f \Vert_{L^s(Y)} \lesssim_{r,s} \Big(\int_{\mathbb{R}^n} \vert f(\ybf) \vert^r w_Y(\ybf) d\ybf\Big)^\frac{1}{r}
    \end{align*}
\end{prop}

Now we return to our exponential sum. 
\begin{equation}
    u(\xbf,t) = \sum_{\qbf \in [Q]^{d}} b_\qbf e\bigl(\xbf \cdot \qbf/Q + t\psi(\qbf/Q) \bigr).
\end{equation}
Fix a dyadic number $\alpha \in 2^\N$ and consider the level set 
\[
\Omega_\alpha = \left\{\xbf \in [0,Q]^{d}: \sup_{0 \leq t \leq Q^2} \vert u(\xbf,t) \vert \in [\alpha,2\alpha)\right\}.
\]
We claim the following level set estimate.
\begin{lemma}\label{lemma:LevelSetEstimateUP} 
We have
\begin{align}
\vert \Omega_\alpha \vert \lesssim_\epsilon \alpha^{-p} Q^{d+\frac{pd}{d+2} + \epsilon} \Vert b_\qbf \Vert_{\ell^2}^p, \hspace{.5cm} 2 \leq p \leq  p_{\crit}.
\end{align}
\end{lemma}

We first show our upper bound in Theorem \ref{thm:UpperMaximalBound} follows from Lemma \ref{lemma:LevelSetEstimateUP}. 
\begin{proof}[Proof of Theorem \ref{thm:UpperMaximalBound}, assuming Lemma \ref{lemma:LevelSetEstimateUP}]
By homogeneity, it suffices to assume $\norm{b_\qbf}_{\ell^2} = 1$. 
We may express the $L^p_{\xbf}$ norm of the maximal function as 
\begin{align*}
    \norm{ \sup_{t \in [0,Q^2]} \vert u(\xbf,t) \vert}_{L_{\xbf}^p ([0,Q]^d)}^p &\leq \sum_{\alpha \in 2^\N}\int_{\Omega_\alpha} \left| \sup_{t \in [0,Q^2]} u(\xbf,t) \right|^p d\xbf +  100Q^d\\
    &\lesssim \sum_{\alpha \in 2^\N} \alpha^p \vert \Omega_\alpha \vert + 100Q^d. 
\end{align*}
where $100Q^d$ is the trivial estimate when $\sup_t \vert u(\xbf,t) \vert \in [0,2)$.

The exponential sum has $Q^d$ terms, so an application of the triangle inequality and Cauchy-Schwarz gives 
\begin{align*}
    \vert u(\xbf,t) \vert \leq \sum_{\qbf \in [Q]^d} \vert b_\qbf \vert \leq Q^{\frac{d}{2}},
\end{align*}
which implies there are $O(\log Q)$ choices of dyadic scales $\alpha$. Applying Lemma \ref{lemma:LevelSetEstimateUP} for each $\alpha \in 2^\N$ gives 
\begin{align*}
    \Big\Vert \sup_{t \in [0,Q^2]}u(\xbf,t) \Big\Vert_{L_{\xbf}^p ([0,Q])^{d}}^p \lesssim \sum _{\alpha \in 2^\N}Q^{d+\frac{pd}{d+2} + \epsilon} \lesssim Q^{d+\frac{pd}{d+2} + \epsilon}, \hspace{0.5cm} 2 \leq p \leq p_{\crit},
\end{align*}
or equivalently
\[
\Big\Vert \sup_{t \in [0,Q^2]} u(\xbf,t) \Big\Vert_{L_{\xbf}^p ([0,Q])^{d}} \lesssim Q^{\frac{d}{p} + \frac{2}{p_{\crit}} + \epsilon}   = Q^{\frac{d}{p} + \frac{d}{d+2} + \epsilon} .
\]
which proves Theorem \ref{thm:UpperMaximalBound}.
\end{proof}
It remains to prove the level set estimate in Lemma \ref{lemma:LevelSetEstimateUP}.

\begin{proof}[Proof of Lemma \ref{lemma:LevelSetEstimateUP}]
Fix a dyadic integer $\alpha \in 2^\N$. For each $\xbf \in \Omega_\alpha$, we pick a $t_\xbf$ such that we have $\vert u(\xbf, t_\xbf) \vert \in [\alpha, 2\alpha)$. Then we have
\begin{align*}
    \vert \Omega_\alpha \vert &\lesssim \alpha^{-p} \int_{\Omega_\alpha} \vert u(\xbf,t_\xbf)\vert^{p}\dxbf.
\end{align*}
By local constancy (Proposition \ref{local constancy}) at $r = p, s = \infty$ to obtain an estimate
\[
    \vert u(\xbf,t_\xbf) \vert^p \lesssim \int_\mathbb{R} \vert u(\xbf, t) \vert^p w_{t_\xbf}(t) dt,
\]
and since $1_{\Omega_\alpha}(\xbf) w_{t_\xbf}(t) \lesssim w_{[0,Q]^d \times[0,Q^2]}(\xbf,t)$, we can bound
\begin{align*}
    \vert \Omega_\alpha \vert &\lesssim \alpha^{-p}\int_{\Omega_\alpha} \int_{\R} \vert u(\xbf,t) \vert^p w_{t_\xbf}(t) dt d\xbf \\
    &\lesssim \alpha^{-p}\int_{\R^d \times \R} \vert u(\xbf,t) \vert^p w_{[0,Q]^d \times[0,Q^2]}(\xbf,t) dtd\xbf \\
    &\lesssim Q^{-d} \alpha^{-p}\int_{\R^d \times \R} \vert u(\xbf,t) \vert^p w_{[0,Q^2]^d \times[0,Q^2]}(\xbf,t) dtd\xbf
\end{align*}
where the last inequality follows from the $Q$-periodicity of the exponential sum $u(\xbf,t)$ in each direction of $\xbf \in \R^d$.
Recall 
\begin{align*}
     u(\xbf,t) = \sum_{\qbf \in [Q]^{d}} b_\qbf e\bigl(\xbf \cdot \qbf/Q + t\psi(\qbf/Q) \bigr).
\end{align*}
has frequencies $(\qbf/Q, \psi(\qbf/Q)) \in \Sigma_{\psi}$, and the points $\qbf/Q$ are separated by distance $\gtrsim Q^{-1}$, hence we apply Theorem \ref{BD-Decoupling} at scale $R = Q^2$ and obtain 
\begin{align*}
    \Vert u(\xbf,t) \Vert_{L^p(w_{[0,{Q^2}]^{d+1}})} \lesssim C_{\epsilon,p,d} Q^\epsilon \Big(\sum_{\theta} \Vert u_\theta \Vert^{2}_{L^p(w_{[0,{Q^2}]^{d+1}})}\Big)^\frac{1}{2}
\end{align*}
in the range $2 \leq p \leq p_{\crit}$, where each frequency cap $\theta$ contains $O(1)$ of the frequencies. Therefore we get
\begin{align*}
    \Vert u \Vert_{L^p(w_{[0,Q^2])^{d+1}})}&\lesssim Q^\epsilon \Bigg(\sum_{\qbf \in [Q]^d} \Big\Vert b_\qbf e\Big(x \cdot \frac{\qbf}{Q} + t \psi\Big(\frac{\qbf}{Q}\Big) \Big) \Big\Vert_{L^p(w_{[0,Q^2]^{d+1}})}^2 \Bigg)^{\frac{1}{2}} \\
    &\lesssim  Q^{\frac{2(d+1)}{p} + \epsilon}\Big(\sum_{\qbf} \vert b_\qbf\vert^2\Big)^{\frac{1}{2}}.
\end{align*}
Taking the $p$-th power, we get the following upper bound for $\vert \Omega_\alpha \vert$
\begin{align}\label{Decouplingupperboundlevelset}
    \vert \Omega_\alpha \vert \lesssim \alpha^{-p} Q^{d+2 + \epsilon} \Vert b_\qbf \Vert_{\ell^2}^p, \hspace{0.5cm} 2 \leq p \leq p_{\crit}.
\end{align}
We split $\vert \Omega_\alpha \vert = \vert \Omega_\alpha \vert^\frac{p}{p_{\crit}} \vert \Omega_\alpha \vert^{1-\frac{p}{p_{\crit}}}$, and use the estimate \eqref{Decouplingupperboundlevelset} on the first factor with $p = p_{\crit}$ and the trivial estimate $\vert \Omega_\alpha \vert \leq Q^d$ on the second factor to obtain 
\begin{align*}
    \vert \Omega_\alpha \vert &\lesssim Q^\epsilon\Big(\alpha^{-p_{\crit}}Q^{d+2} \Vert b_\qbf \Vert_{\ell^2}^{p_{\crit}} \Big)^{\frac{p}{p_{\crit}}}Q^{d-\frac{dp}{p_{\crit}}} = \alpha^{-p} Q^{d + \frac{2p}{p_{\crit}} + \epsilon} \Vert{b_\qbf}\Vert_{\ell^2}^p,
\end{align*}
and recalling $p_{\crit} = \frac{2(d+2)}{d},$ we recover the desired exponent $d +\frac{2p}{p_{\crit}} = d+\frac{pd}{d+2}$ which proves Lemma \ref{lemma:LevelSetEstimateUP}.
\end{proof}

\section{Lattice Points Near Manifolds}\label{section:counting}

Throughout, let us assume $n \ge 3$ and choose $1 \le  d,m \le n-1$ with $d+m = n$. Our goal in this section is to establish Theorem \ref{thm:PointsOnManifolds}. Define the projections
\begin{equation*}
    \pi_{\hor} \colon \R^{d} \times \R^{m} \to \R^{d} \quad \mathrm{and} \quad \pi_{\ver} \colon \R^{d} \times \R^{m} \to \R^{m}.
\end{equation*}
Given a compact convex domain $\Omega \subseteq \R^{d}$, we let $C^2(\Omega;\R^m)$ denote the space of $C^2$ functions $\Omega \to \R^m$. This space is equipped with the norm
\begin{equation*}
    \norm{\phi}_{C^2} = \max_{0 \le |\alpha| \le 2} \sup_{x \in \Omega} |D^\alpha \phi(x)|.
\end{equation*}

\subsection{Moments of Siegel transforms}

Write $G = \SL_{n}(\R)$, $\Gamma = \SL_{n}(\Z)$. Then ${G/\Gamma}$ may be identified with the space of unimodular lattices in $\R^{n}$ via the map $g\Gamma \mapsto g\Z^{n}$. This space is equipped with a canonical $G$-invariant probability measure, which we denote by $\mu$.

Suppose $f \colon \R^n \to \C$ is measurable. Its \textit{Siegel transform} $\hat{f}^{\Sie} \colon {G/\Gamma} \to \C$ is defined by
\begin{equation*}
    \hat{f}^{\Sie}(\Lambda) = \sum_{\vbf \in \Lambda_{\prim}} f(\vbf).
\end{equation*}
Note that, if $f = \ind_S$ for a measurable set $S$, then $\hat{f}^{\Sie} = |\Lambda_{\prim} \cap S|$ is the primitive lattice-point counting function.

\begin{thm}[\cite{Sie}]
    If $f \in C_c(\R^n)$, then
    \begin{equation*}
        \int_{G/\Gamma} \hat{f}^{\Sie} \dmu = \frac{1}{\zeta(n)}\int_{\R^{n}}f\dx.
    \end{equation*}
\end{thm}

\begin{rmk}
    In view of the Riesz representation theorem, all results stated about continuous compactly supported functions also extend to general bounded and compactly supported functions, and in particular to indicator functions of sets. We will frequently make use of this fact without mention.
\end{rmk}

\begin{cor}
    If $S \subseteq \R^n$ is a bounded measurable set, then
    \begin{equation*}
        \E[\hat{\ind}_S^{\Sie}] = \frac{1}{\zeta(n)}\Vol(S).
    \end{equation*}
\end{cor}

\begin{thm}[\cite{Rog}]\label{thm:Rog}
    If $f \in C_c(\R^n)$, then
    \begin{equation*}
        \int_{G/\Gamma} (\hat{f}^{\Sie})^2 \dmu = \frac{1}{\zeta(n)^2} \left(\int_{\R^{n}} f(x)\dx \right)^2 + \frac{1}{\zeta(n)}\int_{\R^{n}}f(x)^2\dx + \frac{1}{\zeta(n)}\int_{\R^{n}}f(x)f(-x)\dx.
    \end{equation*}
\end{thm}

\begin{cor}
    If $S \subseteq \R^n$ is a bounded measurable set, then
    \begin{equation*}
        \E[(\hat{\ind}_S^{\Sie})^2] = \frac{1}{\zeta(n)^2}\Vol(S)^2 + \frac{1}{\zeta(n)}\Vol(S) + \frac{1}{\zeta(n)}\Vol(S \cap -S).
    \end{equation*}
\end{cor}

\begin{cor}\label{cor:variancebound}
    If $S \subseteq \R^{n}$ is bounded measurable with $\Vol(S \cap -S) = 0$ and $F = \hat{\ind}_S^{\Sie}$, then
    \begin{equation*}
        \Var(F) = \E[F].
    \end{equation*}
\end{cor}

\subsection{Points near manifolds}

From now on, we let $\Omega_0 = [0,1]^{d}$. Let $\Dc = \{\theta\}$ denote the standard dyadic partition of $\Omega_0$ into closed cubes of side lengths $2^{-\N}$. We write $\ell(\theta)$ for the side length of a cube $\theta \in \Dc$, and $|\theta| = \ell(\theta)^{d}$ for its $d$-dimensional volume. Write
\begin{equation*}
    \Dc^\ell = \{\theta \in \Dc : \ell(\theta) = 2^{-\ell}\}
\end{equation*}
for the level-$\ell$ subpartition.

For each $\ell \in \N$, we let $\Theta^\ell \subseteq \Dc^\ell$ be a maximal $2^{-\ell}$-separated subcollection; we note that
\begin{equation}\label{eq:separatedcollectioncount}
    |\Theta^\ell| \ge 2^{-d}|\Dc^\ell| = 2^{-d}2^{d\ell}.
\end{equation}

Choose some $\phi \in C^2(\Omega_0;\R^m)$. For each $\epsilon > 0$, each $\ell \in \N$ and each $\theta \in \Dc^\ell$, we define the \textit{thickenings}
\begin{equation*}
    \Mc_\phi(\epsilon,\theta) = \left\{ (\xi,\phi(\xi) + y) : \xi \in \theta,\ y \in \left[0,\epsilon \ell(\theta)^{2}\right]^m \right\}.
\end{equation*}
Let us additionally set
\begin{equation*}
    \alpha = \frac{n}{d+2m} \quad \mathrm{and} \quad \beta = \frac{dn}{d+2m}.
\end{equation*}

\begin{prop}\label{prop:pointsnearmanifold}
    Fix $0 < \kappa < 1$ and $0 < \epsilon < (\kappa/10)^{1/m}$. Moreover set $Q = 2^{\ell/\alpha}$, and assume $Q \ge (5 \cdot 2^d)^{1/\beta}(\kappa \epsilon^m)^{-1/\beta}$. Then
    \begin{equation*}
        \P\left(\big|\bigl\{ \theta \in \Theta^\ell : \Mc_\phi(\epsilon,\theta) \cap Q^{-1}\Lambda_{\prim} \ne \emptyset \bigr\}\bigr| \ge \frac{1}{40 \cdot 2^d}\kappa \epsilon^m Q^{\beta} \right) \ge 1-\kappa.
    \end{equation*}
\end{prop}
\begin{proof}
    To start, we compute
    \begin{equation*}
        \Vol\left( \Mc_\phi(\epsilon,\theta) \right) = |\theta|(\epsilon \ell(\theta)^{2})^m = \epsilon^m2^{-(d+2m)\ell} = \epsilon^m Q^{-n}.
    \end{equation*}
    Therefore, we have
    \begin{equation*}
        \Vol(Q\cdot \Mc_\phi(\epsilon,\theta)) = \epsilon^m.
    \end{equation*}
    Now, let us define the following functions:
    \begin{equation*}
        f_\theta = \ind_{Q\Mc_\phi(\epsilon,\theta)}, \quad f = \sum_{\theta \in \Theta^\ell} f_\theta, \quad F_\theta = \hat{f}_\theta^{\Sie}, \quad F = \sum_{\theta \in \Theta^\ell} F_\theta.
    \end{equation*}
    We write $\zeta = \zeta(n)$. By Siegel's mean value theorem, we have
    \begin{equation}\label{eq:Fexpectations}
        \E[F_\theta] = \zeta^{-1}\epsilon^m \quad \mathrm{and} \quad \E[F] = \zeta^{-1}\epsilon^m|\Theta^\ell|.
    \end{equation}
    By Corollary \ref{cor:variancebound}, since $\Omega_0 \cap -\Omega_0$ has zero $m$-volume, we have
    \begin{equation*}
        \Var(F_\theta) = \E[F_\theta] = \zeta^{-1}\epsilon^m.
    \end{equation*}
    By the second moment method and the inequality $\frac{1}{1+t^{-1}} \ge t(1-t)$, it follows that
    \begin{equation}\label{eq:Ftheta positive}
        \P(F_\theta > 0) \ge \frac{\E[F_\theta]^2}{\E[F_\theta]^2 + \Var(F_\theta)} = \frac{1}{1+\zeta\epsilon^{-m}} \ge \zeta^{-1}\epsilon^m\left( 1 - \zeta^{-1}{\epsilon^m}\right).
    \end{equation}
    Let us set $Z_\theta = \ind_{F_\theta > 0}$ and $Z = \sum_\theta Z_\theta$. Then \eqref{eq:Ftheta positive} and \eqref{eq:Fexpectations} imply
    \begin{equation}\label{eq:Z expectation}
        \E[Z] \ge \zeta^{-1}\epsilon^m\left( 1 - \zeta^{-1}{\epsilon^m}\right)|\Theta^\ell| = \left( 1 - \zeta^{-1}{\epsilon^m}\right)\E[F].
    \end{equation}
    By Paley-Zygmund, we then have for all $0 < u < 1$ that
    \begin{equation*}
        \P(Z \ge u\E[Z]) \ge (1-u)^2 \frac{\E[Z]^2}{\E[Z^2]} \ge (1-u)^2 \frac{\E[Z]^2}{\E[F^2]} \ge (1-u)^2\frac{\E[Z]^2}{\E[F]^2 + \Var(F)}.
    \end{equation*}
    But now, applying \eqref{eq:Z expectation}, and the fact that $\Var(F) = \E[F]$ by Corollary \ref{cor:variancebound}, we see that
    \begin{equation}\label{eq:PaleyZygmund}
        \frac{\E[Z]^2}{\E[F]^2 + \Var(F)} \ge (1-\zeta^{-1}\epsilon^m)^2\frac{\E[F]^2}{\E[F]^2 + \Var(F)} = (1-\zeta^{-1}\epsilon^m)^2\frac{1}{1+\E[F]^{-1}}.
    \end{equation}
    Using the inequality $\frac{1}{1+x} \ge 1-x$, we then obtain
    \begin{equation*}
        \frac{1}{1+(\zeta^{-1}\epsilon^m|\Theta^\ell|)^{-1}} \ge 1 - \frac{\zeta}{\epsilon^m|\Theta^\ell|}.
    \end{equation*}
    In total, using $1 \le \zeta \le 2$ uniformly in $n$, we see that
    \begin{equation}\label{eq:Z positive}
        \P(Z \ge u\E[Z]) \ge (1-2u)(1-2\epsilon^m)\left( 1 - \frac{2}{\epsilon^m|\Theta^\ell|} \right).
    \end{equation}
    Thus, if we choose $u = \kappa/10$, restrict to $0 < \epsilon \le (\kappa/10)^{1/m}$, and assume $|\Theta^\ell| \ge 5\kappa^{-1}\epsilon^{-m}$, then we obtain
    \begin{equation*}
        \P\left(Z \ge \frac{\kappa}{10}\E[Z]\right) \ge 1-\kappa.
    \end{equation*}
    In view of \eqref{eq:separatedcollectioncount}, we see that $|\Theta^\ell| \ge 5\kappa^{-1}\epsilon^{-m}$ holds when $Q \ge (5 \cdot 2^d \kappa^{-1}\epsilon^{-m})^{1/\beta}$. Thus, we have
    \begin{align*}
        \P\left(Z \ge \frac{\kappa}{10}\E[Z]\right) &\le \P\left(Z \ge \frac{\kappa}{20}\E[F]\right) \tag*{by \eqref{eq:Z expectation}} \\
        &\le \P\left(Z \ge \frac{\kappa}{40}\epsilon^m |\Theta^\ell|\right) \tag*{by \eqref{eq:Fexpectations}} \\
        &\le \P\left(Z \ge \frac{\kappa}{40 \cdot 2^d}\epsilon^m Q^{\beta}\right).
    \end{align*}
    Finally, we note that
    \begin{equation*}
        Z = \big|\bigl\{ \theta \in \Theta^\ell : Q \cdot \Mc_\phi(\epsilon,\theta) \cap \Lambda_{\prim} \ne \emptyset \bigr\}\bigr| = \big|\bigl\{ \theta \in \Theta^\ell : \Mc_\phi(\epsilon,\theta) \cap Q^{-1}\Lambda_{\prim} \ne \emptyset \bigr\}\bigr|.
    \end{equation*}
\end{proof}

\subsection{Perturbations}

Suppose $\phi,\kappa,\epsilon,Q$ satisfy the conditions of Proposition \ref{prop:pointsnearmanifold}. We define the set of \textit{good lattices} $\Gc_{\phi,\kappa,\epsilon,Q} \subseteq G/\Gamma$ by
\begin{equation*}
    \Gc_{\phi,\kappa,\epsilon,Q} = \left\{ \Lambda \in G/\Gamma : \big|\bigl\{ \theta \in \Theta^\ell : \Mc_\phi(\epsilon,\theta) \cap Q^{-1}\Lambda_{\prim} \ne \emptyset \bigr\}\bigr| \ge \frac{1}{40 \cdot 2^d}\kappa \epsilon^m Q^{\beta} \right\}.
\end{equation*}
By definition, we have $\mu(\Gc_{\phi,\kappa,\epsilon,Q}) \ge 1-\kappa$.

\begin{lemma}\label{lemma:bumpfunctionderivatives}
    There exists $\chi \in C^\infty(\R^d;\R^1)$ which is radially symmetric, radially decreasing, satisfies $\chi(0) = 1$, $\supp \chi \subseteq \B_1$, and with
    \begin{equation*}
        \norm{\chi}_{C^2} \le 25.
    \end{equation*}
\end{lemma}
\begin{proof}
    Let $u(t) = \exp(1 + \frac{1}{t^2-1})\ind_{|t| < 1}$ and $\chi(\xi) = u(|\xi|)$. Check numerically that $\norm{\chi}_{C^2} \le 25$.
\end{proof}

\begin{prop}\label{prop:perturbations}
    Suppose $\Lambda \in \Gc_{\phi,\kappa,\epsilon,Q}$. Then there exists $\phi_\Lambda \in C^2(\Omega_0;\R^m)$ with
    \begin{equation*}
        \norm{\phi - \phi_\Lambda}_{C^2} \le 100\sqrt{m}\epsilon,
    \end{equation*}
    and
    \begin{equation*}
        \left| \Mc_{\phi_\Lambda} \cap Q^{-1}\Lambda_{\prim} \right| \ge \frac{1}{40 \cdot 2^d}\kappa \epsilon^mQ^\beta.
    \end{equation*}
\end{prop}
\begin{proof}
    For each $\theta \in \Theta^\ell$ choose (if possible) some $\vbf \in Q^{-1}\Lambda_{\prim} \cap \Mc_\phi(\epsilon,\theta)$; by the assumption that $\Lambda$ is good, we may choose at least $\frac{1}{40 \cdot 2^d}\kappa\epsilon^m Q^\beta$ such vectors $\vbf$. For each $\vbf$, let us define a bump function $\chi_\vbf \colon \R^d \to [0,\infty)$ by
    \begin{equation*}
        \chi_\vbf(\xi) = \chi \left( 2Q^\alpha (\xi-\pi_{\hor}(\vbf)) \right),
    \end{equation*}
    where $\chi$ is as in Lemma \ref{lemma:bumpfunctionderivatives}. Since $\Theta^\ell$ is $Q^{-\alpha}$-separated, the functions $\chi_\vbf$ have pairwise disjoint supports. Moreover, we have by construction that
    \begin{equation*}
        \norm{\chi_\vbf}_{C^2} \le 25 \cdot 4 Q^{2\alpha}.
    \end{equation*}
    Let us now define the perturbations $\eta_\Lambda \in C^2(\Omega_0;\R^m)$ by
    \begin{equation*}
        \eta_\Lambda = \sum_{\vbf} \chi_\vbf \cdot a_\vbf,
    \end{equation*}
    where $a_\vbf = \pi_{\ver}(\vbf) - \phi(\pi_{\hor}(\vbf))$. By construction of the thickenings $\Mc_\phi(\epsilon,\theta)$, we see that
    \begin{equation*}
        |a_\vbf| \le \sqrt{m}\epsilon Q^{-2\alpha},
    \end{equation*}
    and so
    \begin{equation*}
        \norm{\eta_\Lambda}_{C^2} \le 100\sqrt{m}\epsilon.
    \end{equation*}
    Moreover, if we set $\phi_\Lambda = \phi + \eta_\Lambda$, then clearly we have $\vbf \in \Mc_{\phi_\Lambda}$ for all $\vbf$, finishing the proof.
\end{proof}

\subsection{Haar measure in exponential coordinates}

Given a matrix $A$, we define its $\ell^\infty$-norm by
\begin{equation*}
    \norm{A}_{\max} = \max_{1 \le i,j \le n}|A_{ij}|.
\end{equation*}
Given a matrix $g \in G$, and $\delta > 0$, we define
\begin{equation*}
    B_\delta^G(g) = \left\{h \in G : \norm{g-h}_{\max} < \delta \right\}.
\end{equation*}

\begin{lemma}\label{lemma:haarmeasureball}
   For all $0 < \delta \ll_n 1$, we have
    \begin{equation*}
        \mu(B_\delta^G(\Id)) \gtrsim_n \delta^{n^2-1}.
    \end{equation*}
\end{lemma}
\begin{proof}
    This follows from the fact that the exponential map $\exp \colon \gfrak \to G$ is a local diffeomorphism with Jacobian 1 at the identity, where $\gfrak = \{X \in \Mat_{n \times n}(\R) : \tr X = 0\}$ is an $n^2-1$-dimensional  $\R$-vector space.
\end{proof}

\subsection{Duality}

Note that if $\Lambda = g\Z^n$, then we would have
\begin{equation*}
    |\Mc_\psi \cap Q^{-1}\Lambda_{\prim}| = |g^{-1}\Mc_{\psi} \cap Q^{-1}\Z_{\prim}^n|.
\end{equation*}
Moreover, since $G$ is unimodular, the distribution of $g^{-1}$ is the same as that of $g$.

Let us choose a fundamental domain $\Fc \subseteq G$ for the $\Gamma$-action, so that each $g \in \Fc$ has minimal $\norm{\cdot}_{\max}$-norm in its $\Gamma$-orbit. Then Proposition \ref{prop:perturbations} implies the following.

\begin{prop}\label{prop:perturbations2}
    Suppose $\phi \in C^2(\Omega_0;\R^m)$, that $0 < \kappa < 1$, that $0 <\epsilon < (\kappa/10)^{1/m}$, and that $Q \ge (5 \cdot 2^d)^{1/\beta}(\kappa\epsilon^m)^{-1/\beta}$. Then there exists a subset $\Gc_{\phi,\kappa,\epsilon,Q}' \subseteq \Fc$ of measure at least $1-\kappa$ so that, for each $g \in \Gc_{\phi,\kappa,\epsilon,Q}'$, there exists $\phi_g \in C^2(\Omega_0;\R^m)$ with $\norm{\phi - \phi_g}_{C^2} \le 100\sqrt{m}\epsilon$, for which
    \begin{equation*}
        |g\Mc_{\phi_g} \cap Q^{-1}\Z_{\prim}^n| \ge \frac{1}{40 \cdot 2^d}\kappa \epsilon^m Q^\beta.
    \end{equation*}
\end{prop}

\subsection{Completing the proof}

\begin{proof}[Proof of Theorem \ref{thm:PointsOnManifolds}]
    Suppose that $0 < \delta \ll_n 1$. Then Lemma \ref{lemma:haarmeasureball} implies
    \begin{equation*}
        \mu\left( B_{2\delta}^G(\Id) \right) \ge \frac{1}{6n}\delta^{n^2-1}.
    \end{equation*}
    Now, choose $\kappa \asymp_n \delta^{n^2-1}$ and assume $0 < \epsilon < (\kappa/10)^{1/m}$. Take $Q \ge (5 \cdot 2^d)^{1/\beta} (\kappa \epsilon^m)^{-1/\beta}$. By Proposition \ref{prop:perturbations2}, the set
    \begin{equation*}
        \Sc_{\phi,\delta,\kappa,\epsilon,Q} = B_{2\delta}^G(\Id) \cap \Gc_{\phi,\kappa,\epsilon,Q}'
    \end{equation*}
    has measure at least $\frac{1}{12n}\delta^{n^2-1}$, and in particular there exists some $g \in \Sc_{\phi,\delta,\kappa,\epsilon,Q}$. By definition, we have $\norm{g-\Id}_{\max} < 2\delta$, and there exists some $\phi_g \in C^2$ with $\norm{\phi-\phi_g}_{C^2} \le 100\sqrt{m}\epsilon$ and
    \begin{equation*}
        |g\Mc_{\phi_g} \cap Q^{-1}\Z_{\prim}^n| \ge \frac{1}{600 \cdot 2^dn} \delta^{n^2-1}\epsilon^m Q^{\frac{dn}{d+2m}}.
    \end{equation*}
\end{proof}

\subsection{Producing graphs}

Note that the manifolds $g\Mc_{\psi}$ as in Theorem \ref{thm:PointsOnManifolds} will have analogous \textit{qualitative} properties to that of $\psi$; in particular, if $\psi$ is uniformly convex, then so is $g\Mc_{\psi}$. However, the manifold $g\Mc_{\psi}$ need not be a graph; we can fix this as follows.

\begin{prop}\label{prop:graphs}
    Suppose we choose $\delta > 0$ sufficiently small depending on $\phi$ in Theorem \ref{thm:PointsOnManifolds}. Then (up to modifying the implicit constants appropriately) we may assume $g\Mc_{\phi_g} \subseteq \Mc_{\psi}$ for some function $\psi \in C^2([0,1]^d;\R^m)$.
\end{prop}
\begin{proof}
    In the proof of Proposition \ref{prop:pointsnearmanifold} and all results which follow, work instead with the subdomain $\Omega_1 = [1/4,3/4]^d$; this will change all implicit constants by a factor depending only on $n$. Now, we choose $\delta > 0$ sufficiently small depending on $\norm{\phi}_{C^0}$, so that
    \begin{equation}\label{eq:lyingoverdomain}
        \pi_{\hor}\left( g \left([1/4,3/4]^d \times [-\norm{\phi}_{C^0}-\epsilon,\norm{\phi}_{C^0}+\epsilon]^m\right) \right) \subseteq [0,1]^d
    \end{equation}
    whenever $\norm{g-\Id}_{\max} < \delta$.

    By continuity of the $G$-action, we have that $g\Mc_{\phi_g}$ is still a graph when $\delta$ is chosen sufficiently small depending on $\phi$. Moreover, by \eqref{eq:lyingoverdomain}, we have $\pi_{\hor}(g\Mc_{\phi_g}) \subseteq \Omega_0$. This completes the proof, by assuming that $\phi$ extends up to a smooth function on, say, $[-1,2]^d$.
\end{proof}

\subsection{Working at all scales}

In this section, we use Proposition \ref{prop:pointsnearmanifold} to produce perturbations at infinitely many scales, simultaneously.

To start, choose $\phi \in C^2(\Omega_0;\R^m)$, let $(\kappa_j)_{j \in \N}$ be a decreasing sequence with $\kappa_0 = \sum_j \kappa_j$, let $0 < \epsilon_j < \kappa_j/10$, and choose a sequence $(\ell_j)$ of integers with
\begin{equation*}
    \ell_j \ge \frac{1}{d}\log_2(\kappa_j^{-1}) + \frac{m}{d}\log_2(\epsilon_j^{-1}) + 1 + \frac{1}{d}\log_2(5).
\end{equation*}
Set $Q_j = 2^{\ell_j/\alpha}$. Then Proposition \ref{prop:pointsnearmanifold} readily implies the following.

\begin{prop}\label{prop:pointsnearmanifoldsallscales}
    In the setup above, there exists $\Gc_\phi \subseteq G/\Gamma$ with $\mu(\Gc_\phi) \ge 1-\kappa_0$, so that for all $\Lambda \in \Gc_\phi$ and all $j \in \N$, we have
    \begin{equation*}
        \left| \left\{ \theta \in \Theta^{\ell_j} : \Mc_\phi(\epsilon_j,\theta) \cap Q_j^{-1}\Lambda_{\prim} \ne \emptyset \right\} \right| \ge \frac{1}{40 \cdot 2^d}\kappa_j\epsilon_j^mQ_j^{\beta}.
    \end{equation*}
\end{prop}

\begin{rmk}
    Note that for each $i \ne j$, we have
    \begin{equation*}
        Q_i^{-1}\Lambda_{\prim} \cap Q_j^{-1}\Lambda_{\prim} = \emptyset,
    \end{equation*}
    since $Q_i^{-1}Q_j > 1$ is an integer when $i < j$. Thus, the points produced at each scale in Proposition \ref{prop:pointsnearmanifoldsallscales} are distinct. However, these points will only be $Q_j^{-1}$-separated, not $Q_j^{-\alpha}$-separated.
\end{rmk}

\begin{ex}
    Let us consider the case of the paraboloid $\Pc \subseteq \R^3$, given by
    \begin{equation*}
        \phi(\xi) = |\xi|^2.
    \end{equation*}
    Let us take $\kappa_j = 2^{-(j+1)}$, $\epsilon_j = \kappa_j/10$, and take $\ell_j  \ge j+5$. Then for at least half of all lattices $\Lambda \in \SL_3(\R)/\SL_3(\Z)$ and all $j \in \N$, we have
    \begin{equation*}
        \left| \left\{ \theta \in \Theta^{\ell_j} : \Mc_\phi(2^{-(j+1)}/10,\theta) \cap 2^{-\frac{4}{3}\ell_j}\Lambda_{\prim} \ne \emptyset \right\} \right| \ge \frac{1}{6400}2^{2(\ell_j-j)}.
    \end{equation*}
\end{ex}

\section{The Case of Curves}\label{section:counting2}

In the proof of Theorem \ref{thm:PointsOnManifolds}, the only place where we used the assumption $n \ge 3$ was in applying Theorem \ref{thm:Rog} to prove Corollary \ref{cor:variancebound}. The analogue of Theorem \ref{thm:Rog} is not true when $n=2$, and though explicit formulas for the second moment are available, they are not as well-behaved. For this reason, we obtain slightly weaker results than in higher dimensions. We maintain the notation of Section \ref{section:counting}, and assume throughout that $n=2$, $d=m=1$.

\subsection{Points near curves}

In this section, our goal is to prove the following analogue of Proposition \ref{prop:pointsnearmanifold}. The method of proof of this subsection is inspired by Athreya and Margulis \cite{AthMar}.

\begin{prop}\label{prop:pointsnearmanifoldn=2}
    Suppose $0 < \epsilon \ll 1$ and $Q \gg_\epsilon 1$. Then
    \begin{equation*}
        \P\left( \left| \left\{ \theta \in \Theta^\ell : \Mc_\phi(\epsilon,\theta) \cap Q^{-1}\Lambda_{\prim} \ne \emptyset \right\} \right| \gtrsim \epsilon Q^{2/3} \right) \ge \frac{1}{64}.
    \end{equation*}
\end{prop}

To establish Proposition \ref{prop:pointsnearmanifoldn=2}, we will need a suitable version of Corollary \ref{cor:variancebound}.

\begin{lemma}[\cite{AthMar}, Lemma 4.6]\label{lemma:AthMar}
     If $f$ is compactly supported of mean zero on $\R^2$, then
     \begin{equation*}
         \norm{\hat{f}_{\prim}^{\Sie}}_{L^2(\SL_2(\R)/\SL_2(\Z))} \le \norm{f}_{L^2(\R^2)}.
     \end{equation*}
\end{lemma}

\begin{thm}[\cite{Ran}, Main Theorem]\label{thm:Ran}
    For each origin-centered ball $B \subseteq \R^2$ with area $A$, we have
    \begin{equation*}
        \int_{\SL_2(\R)/\SL_2(\Z)} \left( \widehat{\ind}_B^{\Sie}(\Lambda) - \frac{A}{\zeta(2)} \right)^2\dmu(\Lambda) = \frac{A}{\zeta(2)} + O(A/\log A).
    \end{equation*}
\end{thm}

Note that Theorem \ref{thm:Ran} does not follow from Lemma \ref{lemma:AthMar}, since (nonzero) constant functions on $\SL_2(\R)/\SL_2(\Z)$ are not Siegel transforms. Using these two results, the following serves as a substitute for Corollary \ref{cor:variancebound} when $n=2$.

\begin{lemma}\label{lemma:varianceboundn=2}
    Suppose $S \subseteq \R^{2}$ is bounded measurable with $\Area(S \cap -S) = 0$, $f = \ind_S$, and $F = \hat{f}^{\Sie}$. Then whenever $\Area(S) \gg 1$, we have
    \begin{equation*}
        \Var(F) < 8\E[F].
    \end{equation*}
\end{lemma}
\begin{proof}
    Let $B \subseteq \R^2$ be the origin-centered ball with $\Area(B) = \Area(S)$, let $\phi = \ind_B$, and let $\Phi = \hat{\phi}^{\Sie}$. Assume that $\Area(S) \gg 1$ is large enough that the big-O term in Theorem \ref{thm:Ran} is of size at most $A/1000$. Then
    \begin{align*}
        \Var(F) &= \norm{F - \E[F]}_{L^2}^2 \\
        &\le \left( \norm{F-\Phi}_{L^2} + \norm{\Phi - \E[F]}_{L^2} \right)^2 \\
        &\le \left( \norm{f-\phi}_{L^2} + \sqrt{\E[F] + \Area(S)/1000} \right)^2 \\
        &\le \left(\sqrt{2\Area(S)} + \sqrt{(1+\zeta(2)/1000)\E[F]} \right)^2 \\
        &= \left(\sqrt{2\zeta(2)} + \sqrt{1+\zeta(2)/1000} \right)^2\E[F] \\
        &< 8\E[F].
    \end{align*}
\end{proof}

\begin{proof}[Proof of Proposition \ref{prop:pointsnearmanifoldn=2}]
    One follows exactly the proof of Proposition \ref{prop:pointsnearmanifold} with Lemma \ref{lemma:varianceboundn=2} in place of Corollary \ref{cor:variancebound}, rescaling by $R \asymp \epsilon^{-1}$ so that the assumption $\Area(RQ \Mc_\phi(\epsilon,\theta)) \gg 1$ is satisfied. Then \eqref{eq:Z expectation} gives
    \begin{equation*}
        \E[Z] \ge \frac{1-o(1)}{8}\E[F].
    \end{equation*}
    By plugging this estimate into the Paley-Zygmund bound \eqref{eq:PaleyZygmund}, we obtain
    \begin{equation*}
        \P(Z \gtrsim \E[Z]) \ge \frac{1}{64}
    \end{equation*}
    whenever $Q \gg_\epsilon 1$.
\end{proof}

\subsection{Transferring to high probability}

Proposition \ref{prop:pointsnearmanifoldn=2} yields only a positive probability estimate, as opposed to a high-probability one. To circumvent this, we can work with the \textit{commensurator} $\SL_2(\Q)$ of $\SL_2(\Z)$.

\begin{thm}\label{thm:PointsOnManifoldsn=2}
    For all $\phi \in C^2([0,1])$, $\delta > 0$, $0 < \epsilon \ll 1$, and $Q \gg_{\epsilon} 1$, there exists $g \in \SL_2(\R)$ with $\max_{i,j}|g_{ij}-\delta_{ij}| \lesssim \delta$, and $\phi_g \in C^2([0,1])$ with $\norm{\phi - \phi_g}_{C^2} \lesssim \epsilon$ for which
    \begin{equation*}
        |g\Mc_{\phi_g} \cap Q^{-1}\Z^2| \gtrsim_{\delta,\epsilon}Q^{2/3}.
    \end{equation*}
\end{thm}
\begin{proof}
    Choose a right-invariant Riemannian metric $\dist_G$ on $G$, i.e., $\dist_G(xg,yg) = \dist_G(x,y)$ for all $x,y,g \in G$, and let $U_\delta^G = \{g \in G : \dist_G(g,\id) < \delta\}$. Note that, on any compact subset $\Kc \subseteq G$, the metric $\dist_G$ is equivalent to the one induced by $\norm{\cdot}_{\max}$, with implied constant depending only on $\Kc$.

    Let $\Kc \subseteq G$ be a compact subset whose projection to $G/\Gamma$ has $\mu$-measure at least $0.99$. By compactness and density of $\SL_2(\Q) \subseteq \SL_2(\R)$, there exists a finite set $\gamma_1,...,\gamma_M \in \SL_n(\Q)$ with $M = O_\delta(1)$ so that
    \begin{equation}\label{eq:finitecover}
        \Kc \subseteq \bigcup_{j=1}^M U_\delta^G \gamma_j.
    \end{equation}
    As in Proposition \ref{prop:perturbations2}, we can find for all $0 < \epsilon \ll 1$ and all $Q \gg_\epsilon 1$ a subset $\Gc_{\phi,\epsilon,Q} \subseteq G$ of measure at least $\frac{1}{64}$ for which we have
    \begin{equation*}
        \left| \left\{ \theta \in \Theta^\ell : \Mc_\phi(\epsilon,\theta) \cap Q^{-1}h\Z^2 \ne \emptyset \right\} \right| \gtrsim \epsilon Q^{2/3}.
    \end{equation*}
    By the measure assumption on $\Kc$, there exists some $h \in \Kc \cap \Gc_{\phi,\epsilon,Q}$. By \eqref{eq:finitecover}, there exists some $\gamma_j$ and some $g \in U_\delta^G$ for which $h=g^{-1}\gamma_j$. Thus, we have
    \begin{equation*}
        \left\{ \theta \in \Theta^\ell : \Mc_\phi(\epsilon,\theta) \cap Q^{-1}h\Z^2 \ne \emptyset \right\} = \left\{ \theta \in \Theta^\ell : g\Mc_\phi(\epsilon,\theta) \cap Q^{-1}\gamma_j\Z^2 \ne \emptyset \right\}.
    \end{equation*}
    Let $H = \mathrm{ht}(\gamma_1,...,\gamma_M)$ be the maximal denominator appearing in any coefficient of the $\gamma_j$, and note $H \lesssim \delta^{-3}$. Then $\gamma_j\Z^2 \subseteq H^{-1}\Z^2$, and so
    \begin{align*}
        \left| \left\{ \theta \in \Theta^\ell : g\Mc_\phi(\epsilon,\theta) \cap Q^{-1}\gamma_j\Z^2 \ne \emptyset \right\} \right| &\le \left| \left\{ \theta \in \Theta^\ell : g\Mc_\phi(\epsilon,\theta) \cap Q^{-1}H^{-1}\Z^2 \ne \emptyset \right\} \right|.
    \end{align*}
    The rest of the proof goes through exactly as in Section \ref{section:counting} (note that we may absorb $H$ into the implicit constants when working with the full graph $\Mc_\phi$, although not with the individual thickenings $\Mc_\phi(\epsilon,\theta)$).
\end{proof}

\nocite{*}
\printbibliography

@article{AthMar,
shorthand = {AM09},
  author    = {Athreya, Jayadev and Margulis, Grigory},
  title     = {Logarithm laws for unipotent flows, {I}},
  journal   = {Journal of Modern Dynamics},
  volume    = {3},
  number    = {3},
  pages     = {359--378},
  year      = {2009}
}

@inproceedings{Ran,
shorthand = {Ran70},
  author    = {Randol, Burton},
  title     = {A group-theoretic lattice-point problem},
  booktitle = {Problems in Analysis: A Symposium in Honor of Salomon Bochner},
  editor    = {Gunning, Robert C.},
  series    = {Princeton Mathematical Series},
  volume    = {31},
  pages     = {291--295},
  publisher = {Princeton University Press},
  address   = {Princeton, NJ},
  year      = {1970}
}

@article{Sie,
shorthand = {Sie45},
  author    = {Siegel, Carl Ludwig},
  title     = {A Mean Value Theorem in Geometry of Numbers},
  journal   = {The Annals of Mathematics},
  volume    = {46},
  number    = {2},
  pages     = {340},
  year      = {1945},
  month     = {apr},
  doi       = {10.2307/1969027},
  url       = {https://www.jstor.org/stable/1969027},
  issn      = {0003486X}
}

@article{Rog,
shorthand = {Rog56},
  author    = {Rogers, Claude Ambrose},
  title     = {The Number of Lattice Points in a Set},
  journal   = {Proceedings of the London Mathematical Society},
  volume    = {s3-6},
  number    = {2},
  pages     = {305--320},
  year      = {1956},
  month     = {apr},
  doi       = {10.1112/plms/s3-6.2.305},
  url       = {http://doi.wiley.com/10.1112/plms/s3-6.2.305},
  issn      = {00246115},
  language  = {en}
}

@article{FuRenWan, title={A note on maximal operators for the Schrödinger equation on $\mathbb{T}^1.$}, url={http://arxiv.org/abs/2307.12870}, DOI={10.48550/arXiv.2307.12870}, note={arXiv:2307.12870}, number={arXiv:2307.12870}, publisher={arXiv}, author={Fu, Yuqiu and Ren, Kevin and Wang, Haoyu}, year={2023}, month=july }

@article{BouDem,
  title={The proof of the {$l^2$} decoupling conjecture},
  author={Bourgain, Jean and Demeter, Ciprian},
  journal={Annals of Mathematics},
  volume={182},
  number={1},
  pages={351--389},
  year={2015},
  publisher={Department of Mathematics of Princeton University}
}

@article{Moy,
  title={Bounds for the maximal function associated to periodic solutions of one-dimensional dispersive equations},
  author={Moyua, Adela and Vega, Luis},
  journal={Bulletin of the London Mathematical Society},
  volume={40},
  number={1},
  pages={117--128},
  year={2008},
  publisher={Oxford University Press},
  doi={10.1112/blms/bdm096}
}

@article{Bar,
  author  = {Barron, Alex},
  title   = {An {$L^4$} Maximal Estimate for Quadratic Weyl Sums},
  journal = {International Mathematics Research Notices},
  year    = {2021},
  volume  = {2022},
  pages   = {17305--17332},
  doi     = {10.1093/imrn/rnab182}
}

@article{DuZha,
  author  = {Du, Xiumin and Zhang, Ruixiang},
  title   = {Sharp {$L^2$} estimates of the Schr\"{o}dinger maximal function in higher dimensions},
  journal = {Annals of Mathematics},
  year    = {2019},
  volume  = {189},
  number  = {3},
  pages   = {837--861},
  doi     = {10.4007/annals.2019.189.3.4}
}

@article{Bak,
  author  = {Baker, Roger},
  title   = {{$L^p$} maximal estimates for quadratic Weyl sums},
  journal = {Acta Mathematica Hungarica},
  year    = {2021},
  volume  = {165},
  number  = {2},
  pages   = {316--325},
  doi     = {10.1007/s10474-021-01173-3}
}

@article{Mia,
  title={Maximal estimates for Weyl sums on $\mathbb{T}^d$},
  author={Miao, Changxing and Yuan, Jiye and Zhao, Tengfei and Barron, Alex},
  journal={Journal of Functional Analysis},
  volume={284},
  number={2},
  pages={109747},
  year={2023},
  publisher={Elsevier},
  doi={10.1016/j.jfa.2022.109747}
}

@article{FuGutMal, title={Decoupling inequalities for short generalized Dirichlet sequences}, volume={16}, ISSN={1948-206X, 2157-5045}, url={https://msp.org/apde/2023/16-10/p05.xhtml}, DOI={10.2140/apde.2023.16.2401}, number={10}, journal={Analysis \& PDE}, author={Fu, Yuqiu and Guth, Larry and Maldague, Dominique}, year={2023}, month=dec, pages={2401–2464}, language={en} }

@article{CaiZha, title={Power loss for the Mizohata-Takeuchi conjecture on $C^k$ convex hypersurfaces}, url={http://arxiv.org/abs/2512.08064}, DOI={10.48550/arXiv.2512.08064}, note={arXiv:2512.08064}, number={arXiv:2512.08064}, publisher={arXiv}, author={Cairo, Hannah and Zhang, Ruixiang}, year={2025}, month=dec }

@incollection{Car,
  author    = {Lennart Carleson},
  title     = {Some analytic problems related to statistical mechanics},
  booktitle = {Euclidean Harmonic Analysis},
  editor    = {J. J. Benedetto},
  series    = {Lecture Notes in Mathematics},
  volume    = {779},
  pages     = {5--45},
  year      = {1980},
  publisher = {Springer},
  address   = {Berlin, Heidelberg},
  doi       = {10.1007/BFb0087666}
}

@incollection{DahKen,
  author    = {Dahlberg, B. E. J. and Kenig, C. E.},
  title     = {A note on the almost everywhere behaviour of solutions to the {Schr{\"o}dinger} equation},
  booktitle = {Harmonic Analysis: Proceedings of a Conference Held at the University of Minnesota, Minneapolis, April 20--30, 1981},
  series    = {Lecture Notes in Mathematics},
  volume    = {908},
  pages     = {205--209},
  year      = {1982},
  publisher = {Springer},
  address   = {Berlin, Heidelberg}
}

@article{Bou,
  author  = {Bourgain, Jean},
  title   = {A note on the {Schr{\"o}dinger} maximal function},
  journal = {Journal d'Analyse Math{\'e}matique},
  year    = {2016},
  volume  = {130},
  number  = {1},
  pages   = {393--396},
  doi     = {10.1007/s11854-016-0035-1},
  publisher={Springer}
}

@article{Pie,
  author  = {Pierce, Lillian B.},
  title   = {On {Bourgain's} counterexample for the {Schr{\"o}dinger} maximal function},
  journal = {The Quarterly Journal of Mathematics},
  year    = {2020},
  volume  = {71},
  number  = {4},
  pages   = {1309--1344},
  doi     = {10.1093/qmath/haaa032},
  publisher = {Oxford University Press}
}

@article{DuGutLi,
  author  = {Du, Xiumin and Guth, Larry and Li, Xiaochun},
  title   = {A sharp {Schr{\"o}dinger} maximal estimate in {$\mathbb{R}^2$}},
  journal = {Annals of Mathematics},
  year    = {2017},
  volume  = {186},
  number  = {2},
  pages   = {607--640},
  doi     = {10.4007/annals.2017.186.2.5}
}

@article{ComLucSta,
  author  = {Compaan, Erin and Luc{\`a}, Renato and Staffilani, Gigliola},
  title   = {Pointwise convergence of the {Schr{\"o}dinger} flow},
  journal = {International Mathematics Research Notices},
  year    = {2021},
  volume  = {2021},
  number  = {1},
  pages   = {596--647},
  doi     = {10.1093/imrn/rnaa036},
  publisher={Oxford University Press}
}

@article{Okt,
  author  = {Öktem, I. Ferit},
  title   = {On the product $\prod_{n=2}^\infty \zeta(n)$},
  journal = {The Mathematical Gazette},
  year    = {2015},
  volume  = {99},
  pages   = {422--426},
  doi     = {10.1017/mag.2015.78}
}

@article{Fai, title={A higher moment formula for the Siegel–Veech transform over quotients by Hecke triangle groups}, volume={15}, ISSN={1661-7207, 1661-7215}, url={https://ems.press/doi/10.4171/ggd/591}, DOI={10.4171/ggd/591},  number={1}, journal={Groups, Geometry, and Dynamics}, author={Fairchild, Samantha}, year={2020}, month=dec, pages={57–81} }

@article{BouDem1,
    author = {Bourgain, Jean and Demeter, Ciprian},
    title = {Decouping for surfaces in $\mathbb{R}^4$},
    journal = {Journal of Functional Analysis},
    year = {2016},
    doi = {https://doi.org/10.1016/j.jfa.2015.11.008.},
    pages = {1299--318},
    volume = {270},
    number = {4}
}

@book{demeter2020fourier,
  author    = {Demeter, Ciprian},
  title     = {Fourier Restriction, Decoupling, and Applications},
  publisher = {Cambridge University Press},
  series    = {Cambridge Studies in Advanced Mathematics},
  volume    = {184},
  year      = {2020},
  doi       = {10.1017/9781108694407},
  isbn      = {9781108499705}
}

@article{bourgain2017study,
  author    = {Bourgain, Jean and Demeter, Ciprian},
  title     = {A study guide for the $\ell^2$ decoupling theorem},
  journal   = {Chinese Annals of Mathematics, Series B},
  volume    = {38},
  number    = {1},
  pages     = {173--200},
  year      = {2017},
  doi       = {10.1007/s11401-016-1066-1},
  eprint    = {1604.06032},
  archivePrefix = {arXiv},
  primaryClass  = {math.CA}
}

@article{And, title={A lower bound for the volume of strictly convex bodies with many boundary lattice points}, volume={106}, ISSN={1088-6850, 0002-9947}, url={https://www.ams.org/tran/1963-106-02/S0002-9947-1963-0143105-7/}, DOI={10.1090/S0002-9947-1963-0143105-7}, number={2}, journal={Transactions of the American Mathematical Society}, author={Andrews, George E.}, year={1963}, pages={270–279}, language={en} }

@article{Kiyohara2024LatticePoints,
  author    = {Daishi Kiyohara},
  title     = {Lattice points on a curve via $\ell^2$ decoupling},
  journal   = {International Journal of Number Theory},
  volume    = {20},
  number    = {08},
  pages     = {2045--2057},
  year      = {2024},
  publisher = {World Scientific Publishing Company},
  doi       = {10.1142/S1793042124500994},
  url       = {https://doi.org/10.1142/S1793042124500994}
}

@article{Demeter2025LevelSet,
  author   = {Ciprian Demeter},
  title    = {Level set estimates for the periodic {S}chr{\"o}dinger maximal function on $\mathbb{T}^1$},
  journal  = {Advances in Mathematics},
  volume   = {467},
  pages    = {110186},
  year     = {2025},
  publisher = {Elsevier},
  doi      = {10.1016/j.aim.2025.110186},
  eprint   = {2402.01099},
  archivePrefix = {arXiv},
  primaryClass = {math.CA}
}
\vspace{\fill}
\noindent
\textsc{Department of Mathematics, University of Wisconsin--Madison, 480 Lincoln Dr., Madison, WI 53706, USA}\\[2pt]
\textit{Email address:} \texttt{gottliebfenv@wisc.edu}\\
\textit{Email address:} \texttt{jtan84@wisc.edu}
\end{document}